\documentclass[11pt]{amsart}
\usepackage{amsmath}
\usepackage{amsmath, amsthm, amscd, amsfonts, amssymb, graphicx, color}
\usepackage{latexsym}
\usepackage{amscd}
\usepackage{enumerate}

\theoremstyle{plain}
\newtheorem{theorem}{Theorem}
\newtheorem{corollary}[theorem]{Corollary}
\newtheorem{lemma}[theorem]{Lemma}

\theoremstyle{remark}
\theoremstyle{definition}
\newtheorem{remark}[theorem]{Remark}
\newtheorem{proposition}[theorem]{Proposition}

\newtheorem{example}[theorem]{Example}

\numberwithin{equation}{section}
\numberwithin{theorem}{section}

\begin{document}

\title{On polynomial functions on  non-commutative groups}

\author{J.~M.~Almira, E.~V.~Shulman}

\subjclass[2010]{22A25, 39A70, 39B52.}

\keywords{Group representations, functional equations, polynomial functions on groups, iterates of difference operators, Montel theorems, compact elements}
%\thanks{$^*$ Corresponding author }
%\let\thefootnote\relax\footnotetext{The research was supported by the Hungarian National Foundation for Scientific Research (OTKA),   Grant No. NK-81402.}

\address{Departamento de Ingenier\'{\i}a y Tecnolog\'{\i}a de Computadores, Facultad de Inform\'{a}tica, Universidad de Murcia, Campus de Espinardo, 30100 Murcia, Spain;\newline
and \newline
Departamento de Matem\'{a}ticas, Universidad de Ja\'{e}n, E.P.S. Linares,  Campus Cient\'{\i}fico Tecnol\'{o}gico de Linares, 23700 Linares, Spain}
\email{jmalmira@um.es; jmalmira@ujaen.es}

\address{Department of Mathematics, Vologda State University, S. Orlova~ 6, Vologda 160000, Russia;\newline
and \newline
Instytut of Mathematics, University of Silesia, Bankowa 14,  PL-40-007 Katowice, Poland}
\email{shulmanka@gmail.com, ekaterina.shulman@us.edu.pl}

\begin{abstract}
Let $G$ be a topological group. We investigate relations between two classes of ``polynomial like''  continuous functions on $G$ defined, respectively, by the conditions
1)  $\Delta_h^{n+1}f=0\ $  for every $h \in G$, \,  and \, 2) $\Delta_{h_{n+1}} \Delta_{h_{n}}\cdots \Delta_{h_{1}}f=0$ \, for every $h_1,\cdots, h_{n+1} \in G.$ \,
It is shown that for many (but not all)   groups these classes coincide. We consider also Montel type versions of the above conditions - when 1) and 2) hold only for $h$ in a generating subset of $G$. Our approach is based on the study of the counterparts of the discussed classes for general representations of groups (instead of the regular representation).
\end{abstract}

\maketitle

\markboth{J.~M.~Almira, E. Shulman}{On polynomial functions on  non-commutative groups}

\section{Introduction}

In the works of M. Frechet \cite{{Frechet-09}, {Frechet-29}},  Van der Lijn \cite{VdL}, S. Mazur and W. Orlicz \cite{M-O} it was shown that the ordinary polynomials on $\mathbb{R}$ or $\mathbb{R}^n$ can be characterized by some conditions (functional equations) which allow one to define and study the analogues of polynomials  on commutative groups, semigroups and linear spaces.

In this work we are concentrated on the study  of two conditions for scalar functions on an {\it arbitrary} topological group $G$ (discrete groups are regarded as a special class of topological ones): the Fr\'{e}chet functional equation
\begin{equation} \label{MD}
\Delta_{h_{n+1}} \Delta_{h_{n}}\cdots \Delta_{h_{1}}f=0,  \qquad \text{ for all \ } h_1,\cdots,h_{n+1}\in G
\end{equation}
and the equation for iterated differences
\begin{equation}\label{UD}
\Delta_h^{n+1}f=0,  \qquad \text{ for all \ } h\in G.
\end{equation}
Here the difference operator $\Delta_h$ is defined as $R_h - 1$ where $R_h$ is the right shift, $R_hf(x)=f(xh)$. It will be shown below (see Remark \ref{left}) that using left shifts one comes to the same classes of functions.

To fix the notations we say that a function $f:G\to \mathbb{C}$ is a {\it polynomial of degree at most $n$} if it satisfies \eqref{MD}; furthermore, a function $f:G\to \mathbb{C}$ is said to be a {\it semipolynomial of degree at most $n$} if it solves \eqref{UD}.
We denote these classes of functions by $(P_n)$ and $(SP_n)$, respectively. The elements of  $(P)=\bigcup_{n\geq 0}(P_n)$ are called {\it polynomials}, and the elements of $(SP)=\bigcup_{n\geq 0}(SP_n)$ are called {\it semipolynomials}. Of course, $(P_n)\subseteq (SP_n)$, for every $n$. In general the inclusion can be strict (see Example \ref{n=1} below). The fact that these two classes coincide for $G=\mathbb{R}$, was established by Mazur and Orlicz  \cite{M-O}. For an arbitrary commutative group $G$, the equality $(SP_n) = (P_n)$ -- not only in the case of complex-valued functions but for maps to an Abelian group with some  restrictions on the latter -- was proved by Van der Lijn \cite{VdL}, Djokovi\'{c} \cite{Dj} (here $G$ can be a semigroup), Szekelyhidi \cite{Szek-McK} (see also an alternative proof in \cite{L_Dj}) and Laczkovich \cite{L} (this paper contains the most general results and systematic treatment of the subject).

    We will find a wider variety  of groups where the equivalence takes place. For example, it will be shown that $(SP_n) = (P_n)$ if $G$ coincides with $G_c$, the closure  of the subgroup generated by all compact elements,  and more generally, if $G/G_c$ is commutative (Corollary~ \ref{comGr}). We will prove that the latter class contains all semisimple and many solvable Lie groups.

The question whether every semipolynomial on an arbitrary group is a polynomial (not necessarily of the same degree), is still open; we answer it (affirmatively) only for semipolynomials of degree 1 (Theorem \ref{1S2}).

One more functional class  related to polynomials is the class (QP) of {\it quasipolynomials} -- the functions whose shifts generate finite-dimensional subspaces.  For $G = \mathbb{R}^d$, quasipolynomials are exactly exponential polynomials, that is the sums of products of polynomials and exponential functions $e^{\langle \lambda, x\rangle}$ with $\lambda\in\mathbb{C}^d$.  So in this case (and in many others, for example for all Lie groups) (QP) is much wider than (P). On the other hand there are examples of groups for which the inclusion $(P)\subset (QP)$ fails (see Section 3). We will show that inside the class of quasipolynomials there is no difference between semipolynomials and polynomials.

In Section 2 we study analogues of the discussed classes for general representations of groups and then in Section 3 we deal with functions on groups, applying general results  to the regular representation.

In Section 4 we consider the following question:
is it sufficient to check the conditions (1) and (2) only for $h_i$ from a topologically generating subset of $G$?

For $G = \mathbb{R}^n$ the question was affirmatively  solved by Montel  \cite{montel_1935, montel} in 1935.
Montel's original proofs were tricky, so that in last few years several simpler proofs have been published \cite{AA_AM,A_NFAO,AK_CJM,AS_AM,AS_DM}. Furthermore, these results have been connected to the theory of local polynomials  \cite{AS_AM}, and some versions of them have been demonstrated also for other classes of functions, e.g., for exponential polynomials  \cite{AS_DM}.

We will prove that the results on polynomials extend to general groups: if  $f$ behaves as a polynomial of degree $\leq m$ when we choose the steps $h_i$ from a generating set of $G$, then $f$ is a polynomial on $G$ of the  same degree. The corresponding statement for commutative $G$ was obtained by M. Laczkovich \cite[Lemma 15]{L}. For semipolynomials the situation is more complicated; we get only some partial results and counterexamples.

\bigskip

\section{Representations}

 In general we consider weakly continuous representations on locally convex topological linear spaces. If $G$ is discrete then a topology of the underlying space is not important and we deal with arbitrary representations on linear spaces.

Let $\pi$ be a representation of $G$ on a vector space $X_{\pi}$. An element $x\in X_{\pi}$ is termed {\it $\pi$-semipolynomial of degree  $\leq n$} if
\begin{equation}\label{same}
(\pi(h) - I)^{n+1}x = 0
\end{equation}
 for any $h\in G$. Furthermore $x\in G_{\pi}$ is called  {\it $\pi$-polynomial of degree $\leq n$} if
\begin{equation}\label{different}
(\pi(h_{n+1}) - I)\ldots (\pi(h_2) - I)(\pi(h_1) - I)x = 0,
\end{equation}
 for any $h_1,...,h_{n+1}\in G$.

We use notations   $(SP_n)_{\pi}$, $(P_n)_{\pi}$  for the corresponding classes, and set $(SP)_{\pi}=\bigcup_{n\geq 0}(SP_n)_{\pi}$,  $(P)_{\pi}=\bigcup_{n\geq 0}(P_n)_{\pi}$. It follows from the definition that $(P_0)_{\pi} = (SP_0)_{\pi} = \Phi_{\pi}$, where $\Phi_{\pi}$ is the set of fixed points for $\pi$: $\Phi_{\pi} = \{x\in X: \pi(g)x = x \text{ for all } g\in G\}$.

\begin{lemma}\label{inv} The sets  $(SP_n)_{\pi}$, $(P_n)_{\pi}$, $(SP)_{\pi}$, and $(P)_{\pi}$ are $\pi$-invariant linear subspaces of $X$; moreover $(SP_n)_{\pi}$ and $(P_n)_{\pi}$ are closed in $X$.
\end{lemma}
\begin{proof} It is obvious that  $(SP_n)_{\pi}$, $(P_n)_{\pi}$, $(SP)_{\pi}$, and $(P)_{\pi}$ are linear subspaces. Since $(SP_n)_{\pi} = \bigcap_{h\in G}\ker (\pi(h) - I)^{n+1}$, it is closed; similarly for $(P_n)_{\pi}$. Furthermore
if $x\in (SP_n)_{\pi}$ then for each $h,g\in G$, $$\pi(g)^{-1}(\pi(h) - I)^{n+1}\pi(g)x = (\pi(g^{-1}hg) - I)^{n+1}x = 0$$ whence $(\pi(h) - I)^{n+1}\pi(g)x = 0$, which implies that
$\pi(g)x \in (SP_n)_{\pi}$. This proves that $(SP_n)_{\pi}$ is $\pi$-invariant. The invariance of $(P_n)_{\pi}$ can be proved in the same way.
\end{proof}

We study the structure of $(SP)_{\pi}$ and $(P)_{\pi}$ for various groups or representations;  in particular, we are interested in the question when  these spaces  coincide.

\begin{theorem}\label{uno} For an arbitrary representation $\pi$ of a group $G$,
\begin{equation}\label{1-2}
(SP_1)_{\pi} \subset (P_2)_{\pi}.
\end{equation}
\end{theorem}
\begin{proof} Let $X_0 = (SP_1)$ and $\rho = \pi|_{X_0}$. For $g\in G$, set $\delta(g) = \rho(g) - 1$.
First we show that for any $g, h \in G$
\begin{equation}\label{anti}
\delta_g\delta_h = -\delta_h\delta_g .
\end{equation}
Indeed, we have, for each $g \in G$,
\begin{equation*}
\rho(g)^2 -2\rho(g) +I=0.
\end{equation*}
Applying to both sides the operator $\rho(g)^{-1}$ we obtain the relation
\begin{equation}\label{relation}
\rho(g) - 2I + \rho(g)^{-1} = 0, \qquad \text{for\ any\ } g \in G.
\end{equation}
Now we write this relation for elements $g$ and $h$ of the group, multiply it, respectively, by $\rho(h)$ and $\rho(g)$ and then sum up:
\begin{equation*}
\rho(h)(\rho(g) - 2I+ \rho(g^{-1})) + \rho(g)(\rho(h)- 2I + \rho(h^{-1}))=0.
\end{equation*}
Rearranging we obtain
\begin{equation}\label{collect}
\begin{split}
(\rho(gh^{-1})+\rho(hg^{-1})-2I) + & (\rho(gh)-\rho(g) -\rho(h)+I)+ \\
& (\rho(hg)-\rho(h)-\rho(g)+I)=0
\end{split}
\end{equation}
In view of (\ref{relation}) the first summand in (\ref{collect}) vanishes and we get the anticommutativity relation (\ref{anti}):
\begin{equation*}
(\rho(g)-I)(\rho(h)-I)+(\rho(h)-I)(\rho(g) -I)=0.
\end{equation*}

Let us consider the set $\mathcal{M} = \{\delta_g,\ g \in G\}$. Since
\begin{equation}\label{main}
\delta_g\delta_h = \delta_{gh} - \delta_g - \delta_h, \qquad \text{for\ any\ } g,h \in G,
\end{equation}
the linear hull $\mathcal{A}$ of $\mathcal{M}$ is an algebra. In view of (\ref{anti}) each two operators $A,B \in \mathcal{A}$  anticommute: $AB = - BA$. Therefore, for each $\ T_1,T_2,T_3 \in \mathcal{A}$,
$$
T_1T_2T_3 = - T_3T_1T_2 = T_2T_3T_1 = - T_1T_2T_3,
$$
which implies that
$$
T_1T_2T_3 = 0, \text{ for any  } T_1,T_2,T_3 \in G.
$$
In particular, $\delta_{g_1}\delta_{g_2}\delta_{g_3} x = 0$,
for all $g_i\in G$, $x\in (SP_1)\pi$. This means that $(SP_1)_\pi \subset (P_2)_{\pi}$.
\end{proof}

\bigskip

\begin{theorem}\label{finite}
If $\dim X_{\pi} < \infty$ then $(SP)_{\pi} =(P)_{\pi}$.
\end{theorem}
\begin{proof} It follows from Lemma \ref{inv}  that the restriction $\rho$ of $\pi$ to $(SP)_{\pi}$ is well defined and acts  by operators with one-element spectra: $\sigma(\rho(g)) = \{1\}$.

Set, as in the proof of Theorem \ref{uno},  $\delta(g) = \rho(g) - I$  and let $\mathcal{A}$ denote the linear span of all $\delta(g)$, $g\in G$. Then $\mathcal{A}$ is an algebra of operators by (\ref{main}).  Since all elements $\delta(g)$ are nilpotent, $trace(\delta(g)) = 0$ and therefore $trace (T) = 0$ for each $T\in L$. It follows that $\mathcal{A}$ is a proper subalgebra of the algebra $\mathcal{L}(X)$ of all linear operators on $X$; by Burnside's Theorem it has non-trivial invariant subspaces (if $\dim X > 1$). Let $X_1\subset X_2$ be two invariant subspaces for $\mathcal{A}$, then the algebra $\mathcal{A}_1$ of all operators induced by operators of $\mathcal{A}$ on $X_2/X_1$ is also linearly generated by nilpotent operators, and the same argument as above shows that either $\dim X_2/X_1 =1$ or there exists a non-trivial invariant subspace for $\mathcal{A}_1$. In other words there exists an invariant subspace for $\mathcal{A}$ between $X_1$ and $X_2$. This shows that $\mathcal{A}$ can be triangularized. Since operators $\delta(g)$ are nilpotent,
all operators in $\mathcal{A}$ are strictly triangular. It follows immediately that $\mathcal{A}^N = 0$ where $N = \dim (SP)_{\pi}$. Therefore $(SP)_{\pi} \subset (P_N)_{\pi}$.
\end{proof}

\begin{corollary}\label{irr}
If $\pi$ is a non-trivial irreducible finite-dimensional representation then $(SP)_{\pi} = \{0\}$.
\end{corollary}

\begin{proof}
The previous proof shows that the restriction of $\pi$ to $(SP)_{\pi}$ is triangularizable.  If $(SP)_{\pi} \neq \{0\}$, then  $(SP)_{\pi} = X_{\pi}$ and $\pi$ is triangularizable; this is possible only if $\dim X_{\pi} = 1$ and $\pi$ is trivial. Thus $\ (SP)_{\pi} = \{0\}$.
\end{proof}

Let us denote by $(F)_{\pi}$ the set of all $x\in X$  such that $\pi(G)x$ is contained in a finite-dimensional subspace. In other words $(F)_{\pi}$ is the union of all invariant finite-dimensional subspaces of $X$.

\begin{corollary}\label{fin}
(i) $(SP)_{\pi}\cap (F)_{\pi} \subset (P)_{\pi}$.

(ii) If $G$ is topologically finitely generated  then $(SP)_{\pi}\cap (F)_{\pi} = (P)_{\pi}$.
\end{corollary}
\begin{proof} (i) If $x\in (SP)_{\pi}\cap (F)_{\pi}$ let $Y$ be a finite-dimensional subspace of $X$ with $x\in Y$. It follows that $x\in (SP)_{\rho}$ where $\rho = \pi|_Y$. By Theorem \ref{finite}, $x\in (P)_{\rho}\subset (P)_{\pi}$.

(ii) It suffices to show that if $G$ is topologically finitely generated and $x\in (P_n)_{\pi}$ then $x\in (F)_{\pi}$. We will do it by induction on $n$. For $n=0$ this is obvious because $(P_0) = \Phi_{\pi}$.  Suppose that it is true for $n\le m-1$, and let $x\in (P_m)_{\pi}$. Then, for every $g\in G$, $(\pi(g) - 1)x\in (P_{m-1})_{\pi} \subset (F)_{\pi}$, in other words each $(\pi(g) - 1)x$ is contained in a finite-dimensional invariant subspace $W(g)$ of $X$. If elements $h_1,...,h_k$ topologically generate $G$ then the subspace $W = \mathbb{C}x + \sum_{i=1}^kW(h_i)$ is finite-dimensional and invariant with respect to all operators $\pi(h_i)$. Therefore $W$ is invariant under all elements of the subgroup $G_0$ of $G$ generated by all $h_i$. Since $G_0$ is dense in $G$ and $W$ is closed, $W$ is invariant for $\pi(G)$. It follows that $x$ is contained in a finite-dimensional invariant subspace, so $x\in (F)_{\pi}$.
\end{proof}

Part (ii) of Corollary \ref{fin} admits an extension to groups which are direct limits of topologically finitely generated groups.

\begin{theorem}\label{dirlim} Let $G = \overline{\cup_{\lambda\in\Lambda}G_{\lambda}}$ where $\lambda\mapsto G_{\lambda}$ is an up-directed net of subgroups, and let each $G_{\lambda}$ be topologically generated by $k_{\lambda}$ elements with $\sup_{\lambda}k_{\lambda} = \nu <  \infty$.  Then $(P)_{\pi} \subset (F)_{\pi}$ and therefore $(SP)_{\pi}\cap (F)_{\pi} = (P)_{\pi}$, for any representation $\pi$  of $G$.
\end{theorem}
\begin{proof}
We have only to prove that $(P)_{\pi}\subset (F)_{\pi}$.  We will show that each $x\in (P_n)_{\pi}$ belongs to an invariant subspace $L$ with $\dim L \le \phi(n,\nu)$, where $\phi$ is some function. Using induction suppose that this is proved for polynomials of degree $\le n-1$. Arguing as in the proof of Corollary \ref{fin}, we fix $\lambda\in \Lambda$, choose the generators $h_1,...,h_{k_{\lambda}}$ of $G_{\lambda}$ and consider the space $W_{\lambda} = \mathbb{C}x + \sum_{i=1}^kW(h_i)$. This space is invariant under $\pi(G_{\lambda})$ and its dimension does not exceed $N := 1 + \nu \phi(n-1,\nu)$. Since the net $\lambda\to W_{\lambda}$ is up-directed, it stabilizes and the subspace $W = \cup_{\lambda\in \Lambda}W_{\lambda}$ is a finite-dimensional $\pi$-invariant subspace. Thus $x\in (F)_{\pi}$ and we may set $\phi(n,\nu) = N$.
\end{proof}

As an example of a direct limit of the above form one can consider the additive group of the field $\mathbb{Q}$. Indeed let $\Lambda = \mathbb{N}$ ordered by the divisibility condition: $m \prec k$ if $m|k$. For each $m\in \Lambda$, let $G_m = \{t\in \mathbb{Q}: mt \in \mathbb{Z}\}$. Then the net $m\mapsto G_m$ is up-directed and $\cup_mG_m = \mathbb{Q}$. Furthermore each $G_m$ is isomorphic to $\mathbb{Z}$ and therefore 1-generated.

\medskip

An element $h$ of a topological group is called {\it compact} if the closure of the cyclic subgroup generated by $h$ is compact. Equivalently $h$ is compact if it is contained in a compact subgroup. Let us denote by $G_c$ the closed subgroup of $G$ generated by all compact elements of $G$. It is clear that $G_c$ is normal.

\begin{lemma}\label{precomp}
$\pi(g)x = x$, for all $g\in G_c$ and $x\in (SP)_{\pi}$.
\end{lemma}
\begin{proof} First let $g$ be compact. Then for every $y\in X^*$, the function $f(g) = \langle\pi(g)x,y\rangle$ is bounded on the closure of $\{g^k: k\in \mathbb{N}\}$, whence the sequence $a_k = \langle\pi(g)^kx,y\rangle$ is bounded. On the other hand, setting $\delta(g) = \pi(g) - 1$ we have that $\delta(g)^nx = 0$, for some $n$, and therefore $a_k = \langle(1 + \delta(g))^kx,y\rangle = \sum_{j=0}^n\binom{k}j\langle\delta(g)^jx,y\rangle$, a polynomial in $k$. It follows easily that $\langle\delta(g)^jx,y\rangle = 0$, for all $j>0$, and in particular,  $\ \langle\delta(g)x,y\rangle = 0$. Since $y$ is arbitrary, $\delta(g) x = 0$, $\pi(g)x = x$.

It follows that $\pi(g)x = x$ for each $g = g_1g_2...g_k$, where all $g_i$ are compact. Since each $g\in G_c$ is the limit of a net of such elements we get that $\pi(g)x = x$ for all $g\in G_c$.
\end{proof}

Note that in many cases $G_c = G$.

\begin{example}\label{semis}
 Let $G$ be a connected complex semisimple Lie group, and $\mathfrak{g}$ be its Lie algebra. To prove that $G = G_c$ we may assume that $G$ is  simply connected, because otherwise $G$ is a quotient of a group with this property (the class of groups topologically generated by compact elements is obviously closed under quotients). Denoting by $H$ the connected component of $G_c$ containing the unit, we have that $H$ is a closed normal subgroup of $G$. By \cite[Theorem II.2.3]{helgason}, $H$ is a Lie group, and its Lie algebra $\mathfrak{h}\subset \mathfrak{g}$ contains  all elements $X\in \mathfrak{g}$ such that ${\rm Exp}(tX)\in H$, for $t\in \mathbb{R}$. If $\mathfrak{g} = \mathfrak{n} + i\mathfrak{n}$ is the Weyl decomposition of $\mathfrak{g}$, then clearly the compact form $\mathfrak{n}$ of $\mathfrak{g}$ is contained in $\mathfrak{h}$ (because ${\rm Exp}(\mathfrak{n})$ consists of compact elements of $G$). Since $H$ is normal, $\mathfrak{h}$ is an ideal of $\mathfrak{g}$, so it is the sum of several simple components of $\mathfrak{g}$. So if $\mathfrak{h}\neq \mathfrak{g}$ then there is a component of $\mathfrak{g}$ which has trivial intersection with the compact form of $\mathfrak{g}$, a contradiction. Thus $\mathfrak{h} = \mathfrak{g}$ and $H = G$, since $G$ is connected and simply connected.

 We conclude that all semipolynomials (and therefore polynomials) on connected semisimple complex Lie groups are constant.
 \end{example}

  \begin{example}\label{solv} Let $G$ be the group of all complex upper triangular matrices $g = (g_{ik})_{i,k\le n}$ with $|g_{ii}| = 1$ for all $i$. If $g_{ii}\neq g_{jj}$ for all $i\neq j$, then $g$ is similar to a diagonal matrix $h$ with $|h_{ii}| = 1$, and therefore $g$ is compact. Therefore compact elements are dense in $G$, and  $G = G_c$.
 \end{example}

\begin{theorem}\label{quotient} Let $X_0 = X^{G_c}$, the space of vectors fixed under all operators $\pi(g)$ where $ g\in G_c$. Then

(i) $X_0$ is invariant under $\pi(G)$ and  there is a  representation $\pi^c$ of $G/G_c$ on $X_0$ such that $\pi(g)x = \pi^c(q(g))x$ for $x\in X_0$, where $q: G\to G/G_c$ is the standard epimorphism.

(ii) An element $x\in X$ belongs to $(SP_n)_{\pi}$ or $(P_n)_{\pi}$ if and only if it belongs to $(SP_n)_{\pi^c}$ or, respectively, $(P_n)_{\pi^c}$.
\end{theorem}

\begin{proof} If $x\in X_0$ then, for all $g\in G$ and $h\in G_c$, one has $\pi(h)\pi(g)x = \pi(g)\pi(g^{-1}hg)x = \pi(g)x$, since $G_c$ is a normal subgroup. Therefore $\pi(g)x\in X_0$, $X_0$ is invariant. If $g_2^{-1}g_1 \in G_c$ then, for some $h\in G_c$, $\pi(g_1)x = \pi(g_2h)x = \pi(g_2)\pi(h)x = \pi(g_2)x$. It follows that setting $\pi^c(gG_c)x = \pi(g)x$ we correctly define a representation of $G/G_c$ on $X_0$ and the needed equality holds (because $q(g) = gG_c$).

Now
\begin{equation*}
\begin{split}
x\in (SP_n)_{\pi} \Leftrightarrow & (\pi(g) - 1)^{n+1}x = 0 \text{ for } g\in G   \Leftrightarrow \\
& (\pi^c(q(g)) - 1)^{n+1}x = 0 \text{ for } g\in G \Leftrightarrow x\in (SP_n)_{\pi^c}.
\end{split}
\end{equation*}
Similarly for polynomials.
\end{proof}

The  results of Van der Lijn, Djokovi\'{c}, Szekelyhidi and Laczkovich mentioned in the Introduction, extend to the representational setting as follows:
\begin{theorem}\label{comm}
If $G$ is commutative then $(SP_n)_{\pi}=(P_n)_{\pi}$ for every representation $\pi$ of $G$.
\end{theorem}
Indeed the proofs in \cite{VdL,Dj,Szek-McK,L} establish in various ways the possibility to write the polynomial $(t_1-1)(t_2-1)...(t_n-1)$ as a sum of polynomials that factorize throw the polynomials of the form $(t_{i_1}...t_{i_k}-1)^n$. The polynomials can be applied to a commutative $n$-tuple of operators $R_{g_i}$ as well as to  $\pi(g_i)$.

\begin{corollary}\label{C10} If $G/G_c$ is commutative then $(P_n)_{\pi}=(SP_n)_{\pi}$, for every representation $\pi$ of $G$.
\end{corollary}

 In particular, the equality  $(P_n)_{\pi}=(SP_n)_{\pi}$ holds if $G$ is generated by a set $M\cup N$ where $M$ is commutative, $N$ consists of compact elements. For example it holds if $G$ is a free product $G_1\ast G_2$ where $G_1$ is compact, $G_2$ is commutative.

  We will finish this section by noting that arguments similar to used in the proof of Lemma \ref{precomp} show that for bounded representations on Banach spaces (in particular, for unitary representations) semipolynomials are precisely fixed points.

 \begin{corollary}\label{bound} If $\pi$ is a bounded representation of a group $G$ on a Banach space $X$ then $(SP)_{\pi} = \{x\in X: \pi(g)x = x \text{  for all  }g\in G\}$.
 \end{corollary}

\bigskip

\section{ Polynomial functions}

Now we return to our initial object -- the representation $R: g \mapsto  R_g$ of a topological group $G$ by right shifts on the space $C(G)$ endowed with the topology of uniform convergence on compacts. Recall that right shifts act by the formula  $R_gf(h) = f(hg)$.

 Clearly $R$-polynomials and $R$-semipolynomials coincide with  {\it polynomials} and {\it semipolynomials} defined in the Introduction  by the equalities (\ref{MD}) and (\ref{UD}). So we write $(SP)$, $(P)$ instead of $(SP)_R$ and $(P)_R$.  We also write $(SP)(G)$ and so on if it is not evident which group we deal with. It is not difficult to show that $(SP)$ and $(P)$ are closed subalgebras of $C(G)$, for any $G$.

It follows also from the definition that the subspace $(F)_R$ of $C(G)$ coincides with the space $(QP)$ of  quasipolynomials.

Recall that a matrix element of a representation $\pi$ is a function of the form $g\mapsto \langle\pi(g)x,y\rangle$ where $x\in X$, $y\in X^*$.
It is known \cite{St} that (continuous) quasipolynomials are exactly the matrix elements of (continuous) finite-dimensional representations.
Among them polynomials  can be  characterized as  matrix elements of a special class of finite-dimensional representations. To introduce it in general, let us say  that a (non-necessarily finite-dimensional) representation $\pi$ of $G$ is an 1-representation if the operator $\pi(g)-1$ is nilpotent for each $g\in G$. Furthermore, for $n\in \mathbb{N}$, we say that $\pi$ is $(1_n)$-representation if $(\pi(h)-1)^{n+1} = 0$ for all $h\in G$. If a more strong condition
\begin{equation}\label{tr-par}
(\pi(h_1)-1)(\pi(h_2)-1)...(\pi(h_{n+1})-1) = 0 \text{ for all } h_1,...,h_{n+1}\in G
\end{equation}
holds then we say that $\pi$ is a $(1_n+)$-representation.

\begin{theorem}\label{char-polyn}
A function $f\in C(G)$ is a polynomial (semipolynomial) of degree at most $n$ if and only if it is a matrix element of a $(1_n+)$-representation (respectively $(1_n)$-representation).
\end{theorem}
\begin{proof}
Assume that $f \in (P_n)$ (the proof in the case $f \in (SP_n)$ is similar). Let $X$ be the closed linear hull in $C(X)$ of all shifts $R_{h}f$, $h\in G$. Let $\pi = R_X$, the restriction of $R$ to $X$. Since $f \in (P_n)_R$ it follows easily that $f\in (P_n)_{\pi}$. Since $(P_n)_{\pi}$ is closed, $\pi$-invariant and $f$ is a cyclic vector of $\pi$, we conclude that $(P_n)_{\pi} = X$, the equality \ref{tr-par} holds, and $\pi$ is a $(1_n+)$-representation.

 Let now $\delta_e$  be the functional of evaluation in the unit $e$ of $G$: $\langle u,\delta_e \rangle = u(e)$, for $u\in C(X)$. Denoting by $\varepsilon$ the restriction of $\delta_e$ to $X$ we have $\langle \pi(g)f,\varepsilon \rangle = \langle R_gf,\delta \rangle = f(g)$.

Conversely, let $\pi$ be a $(1_n+)$-representation of $G$ on a space $X$, and $f(g) = \langle \pi(g)x,\zeta \rangle$ where $x\in X, \zeta \in X^*$. Then $\Delta_hf(g) = \langle \pi(g)(\pi(h) - 1)x, \zeta \rangle$ and therefore
$$\Delta_{h_1}\Delta_{h_2}...\Delta_{h_{n+1}}f(g) = \langle \pi(g)(\pi(h_{n+1}) - I)\ldots (\pi(h_2) - I)(\pi(h_1) - I)x, \zeta \rangle,$$
for all $n$ and all $h_1,...,h_{n+1} \in G$. Since $\pi$ is a $(1_n+)$-representation,
the equality (\ref{different}) holds for all $x\in X$ and $h_1,...,h_{n+1}\in G$. It follows that $\Delta_{h_1}\Delta_{h_2}...\Delta_{h_{n+1}}f(g) = 0$, $f\in (P_n)$.
\end{proof}

Let us call a subset $W$ of $C(G)$ {\it symmetric} if for each $f\in W$, the function $f^*(g):= f(g^{-1})$ belongs to $W$. Note that if $f$ is a matrix element of a representation $\pi$ then $f^*$ is a matrix element of the dual representation $\rho$: $\rho (g) = (\pi (g^{-1}))^*$. Since the representation dual to an $(1_n)$- or $(1_n+)$-representation belongs to the same class, we obtain, applying Theorem \ref{char-polyn}, the following result:
\begin{corollary}\label{symmetry}
For any $n \in \mathbb{N}$, the sets $(SP_n)$ and $(P_n)$ are symmetric subspaces of $C(G)$, and therefore $(SP)$ and $(P)$ are symmetric subalgebras.
\end{corollary}

\begin{remark}\label{left} One could define polynomials and semipolynomials on a group via left regular representation $h \mapsto L_h$ where $L_hf(g) = f(h^{-1}g)$. But it would not change the classes. Indeed, the relation $R_{h}(f^*) = (L_{h^{-1}} f)^*$ implies that if $f\in (P_n)_L$ then $f^* \in (P_n)_R = (P_n)$. So by Corollary \ref{symmetry}, $f\in (P_n)$.
\end{remark}

\begin{theorem}\label{1-repr} For a function $f\in C(G)$, the following conditions are equivalent:

(i) $f \in (P)\cap (QP)$;

(ii) $f$ is a matrix element of a finite-dimensional 1-representation.
\end{theorem}
\begin{proof}
$(i) \Rightarrow (ii)$. Let $X$ be the linear hull in $C(X)$ of all shifts $R_{h}f$, $h\in G$. By our assumptions $\dim X < \infty$. Let $\pi = R_X$, the restriction of $R$ to $X$. It follows easily from the definition that $f\in (P)_{\pi}$. Since $(P)_{\pi}$ is $\pi$-invariant and $f$ is a cyclic vector of $\pi$, $(P)_{\pi} = X$, all $\pi(g)-1$ are nilpotent, $\pi$ is a 1-representation.

 Let now $\delta_e$  be the functional of evaluation in the unit $e$ of $G$: $\langle u,\delta \rangle = u(e)$, for $u\in C(X)$. Denoting by $\varepsilon$ the restriction of $\delta_e$ to $X$ we have $\langle \pi(g)f,\varepsilon \rangle = \langle R_gf,\delta \rangle = f(g)$.

$(ii) \Rightarrow (i)$. Let $\pi$ be a 1-representation of $G$ on a finite-dimensional space $X$, and $f(g) = \langle \pi(g)x,\zeta \rangle$ where $x\in X, \zeta \in X^*$. Then $\Delta_hf(g) = \langle \pi(g)(\pi(h) - 1)x, \zeta \rangle$ and therefore
$$\Delta_{h_1}\Delta_{h_2}...\Delta_{h_{n+1}}f(g) = \langle \pi(g)(\pi(h_{n+1}) - I)\ldots (\pi(h_2) - I)(\pi(h_1) - I)x, \zeta \rangle,$$
for all $n$ and all $h_1,...,h_{n+1} \in G$. It follows from the proof of Theorem \ref{finite} that if $\pi$ is a 1-representation, then there is $n$ such that
the equality (\ref{different}) holds for all $x\in X$ and $h_1,...,h_{n+1}\in G$. Therefore $\Delta_{h_1}\Delta_{h_2}...\Delta_{h_{n+1}}f(g) = 0$, $f\in (P)\cap (QP)$.
\end{proof}

Applying Corollary \ref{fin} to $R$ we get
\begin{theorem}\label{ma} $$(SP)\cap (QP)\subseteq (P)$$
\end{theorem}

In general not every polynomial is a quasipolynomial. For example, if $H$ is an infinite-dimensional Hilbert space, then $f(x)=\langle x,x\rangle$ is  a polynomial of degree $\leq 2$, but the span of its right shifts is infinite-dimensional, so  $f$ is not a quasipolynomial.  A complete characterization of Abelian groups on which $(P) = (QP)$ is given in \cite{laszlo_spectral} (see also \cite{Laszlo-discrete}). We have the following consequence of Corollary \ref{fin}:

\begin{corollary} \label{NCQP}
If $G$ is topologically finitely generated then $(P) = (QP)\cap (SP)$.
\end{corollary}

Note that the class of groups on which every polynomial is a quasipolynomial is strictly larger than the class of topologically finitely generated groups.

A locally compact group is called maximally almost periodic (MAP) if its finite-dimensional unitary continuous representations separate points. Since each connected MAP group is topologically isomorphic to $\mathbb{R}^n \times K$ for some compact connected group $K$ (see, for example, \cite[Theorem 3 in Section 9]{Morris}) then Theorem \ref{quotient} implies
\begin{corollary} If $G$ is a connected MAP group then $(P) = (QP)\cap  (SP)$.
\end{corollary}

Many examples can be constructed using Lemma \ref{precomp}: if $G$ is topologically generated by $G_c$ then each (semi)polynomial on $G$ is a constant. Another set of examples are the direct limits considered in Theorem \ref{dirlim}. In particular,
\begin{corollary}\label{ration}
Let $G = \mathbb{Q}$, the additive group of rational numbers. Each polynomial on $G$ is a quasipolynomial.
\end{corollary}

It should be said that quasipolynomials on $\mathbb{Q}$ have very complicated structure -- even the group of characters (they are quasipolynomials of order 1) of $\mathbb{Q}$ cannot be described in ``elementar'' way. The next statement shows that for polynomials the situation is much better.

\begin{proposition}\label{rat}
 Each polynomial (= semipolynomial) on $\mathbb{Q}$ is the restriction to $\mathbb{Q}$ of an ordinary polynomial (an element of $\mathbb{C}[X]$).
\end{proposition}

\begin{proof} Let $f$ be a polynomial on $\mathbb{Q}$. For each $m\in \mathbb{N}$, let $G_m = \{x\in \mathbb{Q}: mx \in \mathbb{Z}\}$. Clearly the restriction of $f$ to $G_m$ is a polynomial on $G_m$.  It is evident that polynomials on $\mathbb{Z}$ are (restrictions of) ordinary polynomials, so since $G_m$ is isomorphic to $\mathbb{Z}$ under the map $x\mapsto mx$, there is an ordinary polynomial $p_m(t)$ such that $f(x) = p_m(x)$ for $x\in G_m$. Since $\mathbb{Z}\subset G_m$, $p_m(t) = p_1(t)$ for $t\in\mathbb{Z}$. Therefore the equality holds for all $t\in \mathbb{R}$; thus $p_1(t) = f(t)$ on $t\in \cup_mG_m = G$.
\end{proof}

Now, we will provide an example showing that in general $(SP_n)$ can differ from $(P_n)$.

\begin{example}\label{n=1} Let us consider the Heisenberg group
\[ G = \left\{ g=
\begin{pmatrix}
1\ & a_g\ & c_g \\
0\ & 1\ & b_g \\
0\ & 0\ & 1
\end{pmatrix}: \ \  a_g, b_g, c_g \in \mathbb{R} \right\}.
\]
Setting
\[
f(g)=a_g b_g -2 c_g
\]
it is easy to check that for any $h,p\in G$,
$\Delta_h\Delta_pf(g)$ is a constant $a_pb_h - a_hb_p$. It follows that
$\Delta_h^2 f = 0$ for each $h$, but $\Delta_p \Delta_h f \neq 0$  in general.
\end{example}

Thus $(SP_1)\not \subseteq (P_1)$, in general. But the following consequence of Theorem \ref{uno} holds:

\begin{corollary}\label{1S2} On an arbitrary group each semipolynomial of degree at most one is a polynomial of degree at most two.
\end{corollary}

Using Corollary \ref{bound} we obtain
\begin{corollary}\label{bounded}
Regardless to the group $G$, every bounded  semipolynomial is constant.
\end{corollary}

Another consequences of Corollary \ref{bound} can be obtained if one considers arbitrary Banach spaces of functions with shift-invariant norms. For example, one has

\begin{corollary}\label{L2} If a semipolynomial belongs to $L^p(G)$, for some $p\geq 1$, then it is constant.
\end{corollary}

From Lemma \ref{precomp} and Corollary \ref{C10} we get

\begin{corollary}\label{comGr} (i) Assume that $G = G_c$ (that is $G$ is the closure of the subgroup generated by all compact elements of $G$). Then each continuous semipolynomial on $G$ is a constant.

(ii) If $G/G_c$ is commutative then all continuous semipolynomials on $G$ are polynomials.
\end{corollary}

One very special case deserves special mention:

\begin{corollary}\label{order}
If $G$ is generated by elements of finite order then each semipolynomial on $G$ is a constant.
\end{corollary}

Let us now turn to  Theorem \ref{quotient}.

For every closed normal subgroup $H$ of $G$, let $\ q^*: C(G/H) \to C(G)\ $ be the map associated with the standard epimorphism $\ q: G\to G/H$. It is easy to see that if $w$ is a (semi)polynomial on $\ G/H$, then $\ q^*(w)$ is a (semi)polynomial on $G$. The $R$-version of Theorem \ref{quotient} states that all (semi)polynomials on $G$ are of this form with $H = G_c$.

\begin{theorem}\label{R-quotient}
If $q: G \to G/G_c$ is the standard epimorphism, then $q^*$ bijectively  maps $(SP)(G/G_c)$ onto $(SP)(G)$ and $(P)(G/C_c)$ onto $(P)(G)$.
\end{theorem}
\begin{proof} We apply Theorem~ \ref{quotient} to the representation $R$. So $X = C(G)$ and  the space $X_0$ is the set of all functions $f\in C(G)$ such that $f(gk) = f(g)$ for all $g\in G, k\in G_c$. Clearly $X_0$ coincides with the image of $C(G/G_c)$ with respect to $q^*$. Since $q^*$ is injective, there is a map $T: X_0 \to C(G/G_c)$ inverse to $q^*$. Then it is easy to check that
$$TR^c(q(h)) = R_{q(h)}T.$$
Indeed by definition $(Tf)(q(g)) = f(g)$ for all $f\in X_0, g\in G$. Therefore $(TR^c(q(h))f)(q(g)) = (R^c(q(h))f)(g) = f(gh)$ and $(R_{q(h)}Tf)(q(g)) = (Tf)(q(g)q(h)) = (Tf)(q(gh))= f(gh)$.

 Thus $T$ defines the similarity of the representation $R^c$ to the right regular representation of the group $G/G_c$. In particular, it maps the polynomials (semipolynomials) on $G/G_c$ to those of the representation $R^c$. But Theorem \ref{quotient} identifies polynomials (semipolynomials) of $R^c$ with those of $G$. This is what we need.
\end{proof}

We will finish this section by two examples that demonstrate the use of Theorem \ref{R-quotient}.

\begin{example} Let $G = GL(n,\mathbb{C})$ the group of all invertible $n\times n$ matrices. Its subgroup $SL(n,\mathbb{C})$ consisting of all matrices with determinant 1 is semisimple and therefore is contained in $G_c$ (see Example \ref{semis}). Clearly each matrix $\lambda 1$ with $|\lambda| = 1$ is compact and therefore also belongs to $G_c$. Therefore $G_c$ contains the subgroup $H$ of all matrices $a$ with $|\det(a)| = 1$. On the other hand the quotient $G/H$ is isomorphic via the map $a \mapsto \log|\det(a)|$ to the group $\mathbb{R}$, so it does not contain non-trivial compact elements. So $H = G_c$ and polynomials (= semipolynomials) $f$ of degree $m$ on $G$ are of the form $f(a) = p_m(\log|\det(a)|)$ where $p_m$ is an ordinary polynomial of degree $m$ on $\mathbb{R}$.
\end{example}

In particular, the case $m=1$ leads to the well known description (see \cite{Golab, Hosszu, Kuch, Kuczma}) of scalar multiplicative functions ($f(XY)=f(X)f(Y)$) on the set of all $n\times n$ matrices: $f(X)=\phi (|det X|)$, where $\phi$ is a multiplicative function on $\mathbb{R}$.

\begin{example} Let $G = \mathcal{T}(n,\mathbb{C})$ the group of upper-triangular $n\times n$ matrices. Its subgroup $H$ of all matrices with $|a_{11}| = |a_{22}| = ... = |a_{nn}| = 1$ has the property $H_c = H$ (see Example \ref{solv}) and therefore is contained in $G_c$. On the other hand $G/H$ is isomorphic to $\mathbb{R}^n$ (via the map $a\mapsto (\log|a_{11}|,...,\log|a_{nn}|)$ so $G_c = H$ and any (semi)polynomial on $G$ is of the form $f(a) = p(\log|a_{11}|,...,\log|a_{nn}|)$, where $p$ is an ordinary polynomial on $\mathbb{R}^n$.
\end{example}

\bigskip

\section{Montel's type statements}

We now consider the situation when  $f\in C(G)$ behaves as a polynomial or as a semipolynomial on generators of $G$. The following theorem extends \cite[Lemma 15]{L} to non-commutative groups.

\begin{theorem}\label{LNC}
Let $G$ be a topological group and let a subset $E\subseteq G$ topologically generate $G$.  If for $f\in C(G)$  and some $m \in \mathbb{N}$,
$$
\Delta_{h_1}\cdots\Delta_{h_m}f=0  \qquad {\rm \ for\ any \ \ } h_1,\ldots,h_m\in E,
$$
 then $f\in (P_{m-1})$.
\end{theorem}
\begin{proof} We should prove that $\ \Delta_{g_1}\cdots\Delta_{g_m}f=0\ $ for all $\ g_1,\ldots, g_m\in G$.

Let us  use induction on $m$. Suppose that for $m-1$ the statement is true. Then since for each $h \in E$, $\Delta_{h_1}\cdots\Delta_{h_{m-1}}(\Delta_h f)=0$, we get
\begin{equation}\label{*}
\Delta_{g_1}\cdots\Delta_{g_{m-1}}(\Delta_h f)=0, \qquad \text{for any } g_1,\ldots, g_{m-1} \in G.
\end{equation}
Let $H$ be the set of those $h\in G$ for which (\ref{*}) holds for every $g_1,\ldots, g_{m-1} \in G.$  Obviously, $H\supset E$.  Furthermore,
\begin{equation}\label{product}
\Delta_{ab}=\Delta_a+\Delta_b+\Delta_a\Delta_b
\end{equation} and
\begin{equation}\label{inverse}
\Delta_{a^{-1}} = -\Delta_a -\Delta_{a^{-1}}\Delta_a.
\end{equation}
It follows that if $a,b \in H$ then
\begin{equation*}
\begin{split}
\Delta_{g_1}\cdots\Delta_{g_{m-1}}\Delta_{ab}f=& \Delta_{g_1}\cdots\Delta_{g_{m-1}}\Delta_{a}f +\Delta_{g_1}\cdots\Delta_{g_{m-1}}\Delta_{b}f+ \\
& \Delta_{g_1}\cdots\Delta_{g_{m-1}}\Delta_{a}\Delta_{b}f=0
\end{split}
\end{equation*}
and, similarly,
$$
\Delta_{g_1}\cdots\Delta_{g_{m-1}}\Delta_{a^{-1}}f = -\Delta_{g_1}\cdots\Delta_{g_{m-1}}\Delta_{a}f - \Delta_{g_1}\cdots\Delta_{g_{m-1}}\Delta_{a^{-1}}\Delta_{a}f =0.
$$
So, $H$ is a subgroup of $G$. It follows immediately from the continuity of $f$ that $H$ is closed. Therefore $H = G$.

The case $m=1$ follows from the same relations (\ref{product}) and (\ref{inverse}): the set of those $h\in G$ for which $\Delta_h f=0$ is a closed subgroup of $G$ containing $E$.
\end{proof}
\begin{remark}\label{non-scalar}
The same proof shows that a similar statement is true if $f$ takes values in arbitrary Abelian group instead of $\mathbb{C}$.
\end{remark}

Let us consider now the Montel problem for semipolynomial relations: does any function satisfying the conditions
 \begin{equation}\label{Msp}
 \Delta_h^{n+1}f=0, \text{ for all }h\in E, \text{ where } E \text{ topologically generates } G
 \end{equation}
belong to $(SP_n)(G)$?

\begin{theorem}\label{MCK}
Let $G$ be topologically generated by a subset $E \subset G$, and let $f\in C(G)$.
Assume that for each $h \in E$, there is $n(h)$ with $ \Delta_h^{n(h)}f=0$.
 If

 (a) $E$ consists of compact elements,

or

 (b) $G$ is commutative and $E$ is finite

then $f \in (SP) = (P)$.
 \end{theorem}
\begin{proof} If (a) holds then it follows from Lemma \ref{precomp} that $R_hf =f$ for all $h\in E$, and this immediately implies that the same holds for all $h\in G$. Therefore $f$ is constant.

 (b) Let $G$ be commutative and $E = \{ h_1, \ \ldots\ h_s\}$. Let us set $m = \sum_{j=1}^sn(h_j)$ and take arbitrary $h_{i_1}\ldots ,h_{i_m} \in E$.
   Since commutativity of $G$ implies commutativity of the difference operators, we have
\[
\Delta_{h_{i_1}}\cdots \Delta_{h_{i_m}}f=\Delta_{h_1}^{\alpha_1}\cdots \Delta_{h_s}^{\alpha_s}f
\]
for some $\alpha_1,\ldots, \alpha_s$ such that $m=\alpha_1+\cdots+\alpha_s$.
It follows that  $\alpha_k\geq n(h_k)$ for some $k$ whence $\Delta_{h_k}^{\alpha_k}f = 0$, and therefore
\[
\Delta_{h_{i_1}}\cdots \Delta_{h_{i_m}}f=0.
\]
The proof ends by using Theorem \ref{LNC}.
 \end{proof}

\medskip

\begin{remark}\label{inf-gen}
The assumption that $G$ is finitely generated cannot be dropped. Indeed, let $\ G = \mathbb{Z}_0^{\infty}$, \, the group of all sequences $\overrightarrow{n} = (n^{(1)},n^{(2)},\ \ldots \ )$ of integers with only finite numbers of non-zero elements. Let $\ E = \{\overrightarrow{e}_i: 1\le i<\infty\}\ $ be the standard set of generators: $\ \overrightarrow{e}_i^{(j)} = \delta_{ij}$. \, Define a function $f$ on $G$ by the formula
$$
f(\overrightarrow{n}) = n^{(1)} + n^{(2)}n^{(3)} + n^{(4)}n^{(5)}n^{(6)} +\ \ldots \ = \sum_{k=1}^{\infty}\prod_{i=N_{k}+1}^{N_{k+1}}n^{(i)},
$$
where $\ N_k = k(k+1)/2$.\, Then it is easy to see that $\ \Delta_{\overrightarrow{e}_i}^2f = 0$, \, for all $i$,\,  but $\ \left(\prod_{i=N_{k}+1}^{N_{k+1}}\Delta_{\overrightarrow{e}_i}\right)f \neq 0$, \, for each $k$, \, and therefore $\ f\notin (P) = (SP)$.
\end{remark}

\medskip

 We saw that the semipolynomial version of the Montel theorem holds for finitely generated commutative groups and for groups generated by compact elements. Now we will show that commutativity modulo $G_c$ is not sufficient: a semipolynomial on generators need not be a semipolynomial (= polynomial) on $G$ even if $G$ is a free product of a finite group and a singly generated group.

\begin{theorem} \label{NonCom}
Let $$G=\langle a,b|\ a^2=e\rangle$$ (the free product of $\mathbb{Z}$ and $\mathbb{Z}_2$). Then there exists a function $f:G\to\mathbb{C}$ such that $\Delta_a^2f=\Delta_b^2f=0$ and $f$ is not a semipolynomial nor a quasipolynomial on $G$.
\end{theorem}
\begin{proof}  Each element of $G$ can be written in the form $(a)b^{n_1}a \ldots ab^{n_k}(a)$ with some non-negative integers $n_i$. We write the first and the last symbols in brackets to show that they can be present or absent.

A function $f$ which is a ``2-semipolynomial on the generators'' satisfies the condition
\begin{equation}\label{exp-2}
\left\{\begin{aligned}
&f(xa^2)=2f(xa) - f(x)\\
&f(xb^2)=2f(xb) - f(x)
\end{aligned} \right.
\end{equation}
Since $a^2=e$, the first equation of the system is equivalent to the relation $f(xa)=f(x)$. It is easy to deduce from the second equation that $f(xb^n) = nf(xb) - (n-1)f(x)$ for every integer $n$.

To construct a function $f$ satisfying (\ref{exp-2}) and which is not a quasipolynomial nor a  semipolynomial, let us denote by $E_k$ the set of all elements $g\in G$ of the form $g = b^{n_1}ab^{n_2}...ab^{n_{k}}$; for $k = 0,1$ we mean $E_0 = \{e\}$, $E_1 = \{b^j: j = 1,...\}$. Then each element of $G$ either belongs to $E_k$, for some $k\ge 1$, or can be written in the form $ah$, $ha$, or $aha$, where $h\in E_k$. We firstly choose a sequence of numbers $\alpha_k$, $k = 1,...$ arbitrarily.

 The function $f$ will be defined inductively as follows. Let $f(e) = 0$. Now if for all $h\in E_k$ the function is defined, then we set $f(ah) = f(ha) = f(aha) = f(h)$, $f(hab^m) = m\alpha_k - (m-1)f(h)$. Now $f$ is defined on $E_{k+1}$. Proceeding in this way we define $f$ on $G$.

 It follows easily from the definition that the conditions (\ref{exp-2}) are satisfied. On the other hand if we choose  $\alpha_k$ growing more quickly than exponents (for example $\alpha_k = k!$) then the orbit of $f$ cannot be contained in a finite-dimensional subspace. Indeed for $h = ab$, we have $f(h^k) = \alpha_k$, but if $f$ is a matrix element then $$R_{h^k}f = \langle \pi(h)^k\xi,\eta\rangle $$ and $\|R_{h^k}f\| \le C \|\pi(h)\|^k$. Since $\varphi\mapsto \varphi(e)$ is a linear functional on the linear span of the orbit, one has that $$|f(h^k)| = |R_{h^k}f (e)| \le C\|R_{h^k}f\|$$ has at most exponential growth. This is a contradiction which proves that $f$ is not a quasipolynomial.

The same argument shows that $f$ is not a semipolynomial. Indeed if $(R_h -1)^n f = 0$, for some $h\in G$,  then the restriction of $f$ to the subgroup generated by $h$ is a quasipolynomial. Hence $f(h^k)$ has a polynomial growth, which contradicts our condition $f(h^k) = \alpha_k$ for $h = ab$.
\end{proof}

In conclusion we obtain a Montel type result for bounded functions.

\begin{corollary}\label{m_f_c} Let a group $G$ be topologically generated by a subset $E \subset G$, and let $f\in C(G)$ be bounded.
If for each $h \in E$, there is $n(h)$ such that
 $\Delta_h^{n(h)}f=0$,
 then $f$ is constant.
\end{corollary}
\begin{proof}
Let $h\in E$, then as in the proof of Lemma \ref{precomp}  we see that for each $g\in G$, the bounded sequence $a_k(g) = (R_{h^k}f)(g) = ((1+\Delta_h)^kf)(g)$ is a polynomial in $k$. Hence it does not depend on $k$ and therefore its coefficients $\binom{k}j\Delta_h^jf(g)$, $1\le j\le n(h)$, are zero. So $\Delta_hf(g) = 0$, $R_hf = f$. Since $E$ generates $G$, $R_gf = f$ for all $g\in G$. Thus $f$ is constant.
\end{proof}

\bigskip

{\bf Acknowledgments:} The authors are grateful to Professor Henrik Stetk{\ae}r for useful discussions and to the Referee for helpful suggestions and information.

\end{document}